\documentclass[a4paper,12pt,twoside]{article}
\usepackage{amssymb,amsmath,amsthm,amsfonts,latexsym}
\usepackage{graphicx,verbatim,caption,subcaption}
\usepackage[pdftex,bookmarks,colorlinks=false]{hyperref}
\usepackage[hmargin=1.25in,vmargin=1.25in]{geometry}
\usepackage[auth-sc-lg,affil-sl]{authblk}
\def\ni{\noindent}

\newtheorem{theorem}{Theorem}[section]

\newtheorem{corollary}[theorem] {Corollary}
\newtheorem{definition}[theorem]{Definition}

\newtheorem{proposition}[theorem]{Proposition}

\title{\bf {\sc Strong Integer Additive Set-valued Graphs: A Creative Review}}

\author{N. K. Sudev\footnote{Corresponding author}}
\affil{\small Department of Mathematics\\ Vidya Academy of Science \& Technology\\Thalakkottukara, Thrissur-680501, Kerala, India.\\E-mail:{\em sudevnk@gmail.com}}

\author{K. A. Germina}
\affil{\small PG \& Research Department of Mathematics\\ Marymatha Arts \& Science College\\Mananthavady, Wayanad-670645, Kerala, India.\\E-mail:{\em srgerminaka@gmail.com}}

\author{K. P. Chithra} 
\affil{\small Naduvath Mana, Nandikkara \\Thrissur-680301, Kerala, India.\\E-email:{\em chithrasudev@gmail.com}}

\date{}

\begin{document}
\maketitle

\begin{abstract}
For a non-empty ground set $X$, finite or infinite, the {\em set-valuation} or {\em set-labeling} of a given graph $G$ is an injective function $f:V(G) \to \mathcal{P}(X)$, where $\mathcal{P}(X)$ is the power set of the set $X$. A set-indexer of a graph $G$ is an injective set-valued function $f:V(G) \to \mathcal{P}(X)$ such that the function $f^{\ast}:E(G)\to \mathcal{P}(X)-\{\emptyset\}$ defined by $f^{\ast}(uv) = f(u ){\ast} f(v)$ for every $uv{\in} E(G)$ is also injective., where $\ast$ is a binary operation on sets. An integer additive set-indexer is defined as an injective function $f:V(G)\to \mathcal{P}({\mathbb{N}_0})$ such that the induced function $g_f:E(G) \to \mathcal{P}(\mathbb{N}_0)$ defined by $g_f (uv) = f(u)+ f(v)$ is also injective, where $\mathbb{N}_0$ is the set of all non-negative integers and $\mathcal{P}(\mathbb{N}_0)$ is its power set. An IASI $f$ is said to be a strong IASI if $|f^+(uv)|=|f(u)|\,|f(v)|$ for every pair of adjacent vertices $u,v$ in $G$. In this paper, we critically and creatively review the concepts and properties of strong integer additive set-valued graphs.
\end{abstract}


\section{Preliminaries}

\subsection{Introduction to Set-Valued Graphs}

For all  terms and definitions, not defined specifically in this paper, we refer to \cite{FH}. For more about graph classes, we further refer to \cite{BLS} and \cite{JAG}. Unless mentioned otherwise, all graphs considered here are simple, finite and have no isolated vertices.

The researches on graph labeling problems commenced with the introduction of the concept of number valuations of graphs in \cite{AR1}. Since then, the studies on graph labeling have contributed significantly to the researches in graph theory and associated filelds. Graph labeling problems have numerous theoretical and practical applications. Many types of graph labelings are surveyed and listed in \cite{JAG}.

Motivated from various problems related to social interactions and social networks, in \cite{BDA1}, Acharya introduced the notion of set-valuation of graphs analogous to the number valuations of graphs. For a non-empty ground set $X$, finite or infinite, the {\em set-valuation} or {\em set-labeling} of a given graph $G$ is an injective function $f:V(G) \to \mathcal{P}(X)$, where $\mathcal{P}(X)$ is the power set of the set $X$. 

Also, Acharya defined a {\em set-indexer} of a graph $G$ as an injective set-valued function $f:V(G) \to \mathcal{P}(X)$ such that the function $f^{\ast}:E(G)\to \mathcal{P}(X)-\{\emptyset\}$ defined by $f^{\ast}(uv) = f(u ){\ast} f(v)$ for every $uv{\in} E(G)$ is also injective, where $\mathcal{P}(X)$ is the set of all subsets of $X$ and $\ast$ is a binary operation on sets. 

Taking the symmetric difference of two sets as the operation  between two set-labels of the vertices of $G$, the following theorem was proved in \cite{BDA1}.

\begin{theorem}\label{T-SIG}
\cite{BDA1} Every graph has a set-indexer.
\end{theorem}

\subsection{Integer Additive Set-Valued Graphs}

\begin{definition}\label{D-IASL}{\rm
\cite{GA} Let $\mathbb{N}_0$ denote the set of all non-negative integers and $\mathcal{P}(\mathbb{N}_0)$ be its power set. An {\em integer additive set-labeling} (IASL, in short) of a graph $G$ is defined as an injective function $f:V(G)\to \mathcal{P}(\mathbb{N}_0)$ which induces a function $f^+:E(G) \to \mathcal{P}(\mathbb{N}_0)$ such that $f^+ (uv) = f(u)+ f(v),~ uv\in E(G)$. A Graph which admits an IASL is called an {\em integer additive set-labeled graph} (IASL-graph).}
\end{definition}

During a personal communication with the second author, Acharya introduced the notion of integer additive set-indexers of graphs, using the concept of sum sets of two sets of non-negative integers and the definition was first appeared in \cite{GA} as follows. 

\begin{definition}\label{D-IASI}{\rm
\cite{GA} An integer additive set-labeling $f:V(G)\to \mathcal{P}(\mathbb{N}_0)$ of a graph $G$ is said to be an {\em integer additive set-indexer} (IASI) if the induced function $f^+:E(G) \to \mathcal{P}(\mathbb{N}_0)$ defined by $f^+ (uv) = f(u)+ f(v)$ is also injective. A Graph which admits an IASI is called an {\em integer additive set-indexed graph} (IASI-graph).}
\end{definition}

An IASL (or IASI) is said to be {\em $k$-uniform} if $|f^+(e)| = k$ for all $e\in E(G)$. That is, a connected graph $G$ is said to have a $k$-uniform IASL (or IASI) if all of its edges have the same set-indexing number $k$. If $G$ is a graph which admits a $k$-uniform IASI and $V(G)$ is $l$-uniformly set-indexed, then $G$ is said to have a {\em $(k,l)$-completely uniform IASI} or simply a {\em completely uniform IASI}.

\begin{theorem}
{\rm \cite{GS0}} Every graph $G$ admits an integer additive set-labeling (integer additive set-indexer).
\end{theorem}

If either $f(u)$ or $f(v)$ is countably infinite, then clearly their sumset will also be countably infinite and hence the study of the cardinality of the sum set $f(u)+f(v)$ becomes trivial. Hence, we restrict our discussion to finite sets of non-negative integers. We denote the cardinality of a set $A$ by $|A|$. The cardinality of the set-label of an element (vertex or edge) of a graph $G$ is called the {\em set-indexing number} of that element. If the set-labels of all vertices of $G$ have the same cardinality, then the vertex set $V(G)$ is said to be {\em uniformly set-indexed}.

Certain studies about integer additive set-indexed graphs have been done in \cite{TMKA}, \cite{GS1}, \cite{GS2} and \cite{GS0}.

Analogous to Theorem \ref{T-SIG}, we have proved the following theorem.

\begin{theorem}\label{T-IASI1}
{\rm \cite{GS0}} Every graph has an integer additive set-labeling (or an integer additive set-indexer).
\end{theorem}

\section{Strong Integer Additive Set-Labeled Graphs}

The cardinality of the sum set of two non-empty finite sets $A$ and $B$ is always less than or equal to the product $|A|\,|B|$. The study regarding the characteristics of graphs in which the set-indexing number of every edge is the product of the set-indexing numbers of its end vertices  arises much interest. Hence, we defined 

\begin{definition}{\rm  
\cite{GS2} If a graph $G$ has a set-indexer $f$ such that $|f^+(uv)|=|f(u)+f(v)|=|f(u)|\,|f(v)|$ for all vertices $u$ and $v$ of $G$, then $f$ is said to be a {\em strong IASI} of $G$. A graph which admits a strong IASI is called a {\em strong IASI-graph}.} 
\end{definition} 

Let $A$ be a non-empty finite set of non-negative integers. The {\em difference set} of $A$, denoted by $D_A$, is the set defined by $D_A=\{|a-b|:a,b\in A\}$. Now, we necessary and sufficient condition for a graph $G$ to admit a strong IASI, in terms of the difference sets of the set-labels of the vertices of a given graph, is established in \cite{GS2} as follows.

\begin{theorem}\label{T-SIASI1}
{\rm \cite{GS2}} A graph $G$ admits a strong IASI if and only if the difference sets of the set-labels of any two adjacent vertices in $G$ are disjoint.
\end{theorem}

Let us use the notation $A<B$ to indicate that the sets $A$ and $B$ are mutually disjoint sets. By the sequence $A_1<A_2<A_3<\ldots <A_n$, we mean that the given sets are pairwise disjoint. The relation $<$ is called the {\em difference relation} of $G$ and the {\em length} of a sequence of difference sets is the number of difference sets in that sequence. Then, we have

\begin{theorem}\label{T-SIASI-Gg} 
{\rm \cite{GS2}} Let $G$ be a connected graph on $n$ vertices. Then, $G$ admits a strong IASI if and only if there exists a finite sequence of difference sets $D_1<D_2<D_3< \ldots <D_m;~ m\le n$, where $D_i$ is the difference set of the set-label a vertex $v_i$ of $G$. 
\end{theorem}

By Theorem \ref{T-SIASI1}, if $G$ admits a strong IASI $f$, then for any two adjacent vertices $u$ and $v$ of $G$, we have $D_{f(u)}<D_{f(v)}$. The relation $<$ forms one or more sequence of difference sets for $G$. Then, we introduce the following notion.

\begin{definition}{\rm
\cite{GS12} The {\em nourishing number} of a set-labeled graph is the minimum length of the maximal sequence of difference sets of the set-labels of the vertices in $G$. The nourishing number of a graph $G$ is denoted by $\varkappa(G)$.}
\end{definition}

In view of the above notions, we proposed a necessary and sufficient condition for a complete graph $G$ as follows.

\begin{theorem}\label{T-SIASI-Kn}
\cite{GS2} A complete graph $K_n$ admits a strong IASL (or IASI) if and only if the difference sets of the set-labels of all vertices of $G$ are pairwise disjoint.
\end{theorem}

The following theorem establishes the hereditary nature of the existence of a strong IASL (or a strong IASI). 

\begin{theorem}\label{T-SIASI-SG}
Any subgraph of a strong IASL-graph also admits a (induced) strong IASI.
\end{theorem}

That is, for a vertex $v_i$ of a complete graph $K_n$, if $D_i$ is the difference set of $f(v_i)$, then we have $D_1<D_2<D_3<\ldots<D_n$. Therefore, 

\begin{proposition}
{\rm \cite{GS12}} The nourishing number of a complete graph $K_n$ is $n$.
\end{proposition} 

\ni The following results are on the nourishing number of different graphs.

\begin{proposition}
{\rm \cite{GS12}} The nourishing number of a bipartite graph is $2$.
\end{proposition}

\begin{proposition}
{\rm \cite{GS12}} The nourishing number of a triangle-free graph is $2$.
\end{proposition}

Recall that a {\em clique} of graph $G$ is a complete subgraph of $G$ and the {\em clique number},denoted by $\omega(G)$, of $G$ is the number of vertices in a maximal clique of $G$. Then, due to the above results we have

\begin{theorem}\label{T-NN1}
{\rm \cite{GS12}} The nourishing number of a graph $G$ is the clique number $\omega(G)$ of $G$.
\end{theorem}

Invoking Theorem \ref{T-NN1}, we have established the following results on various operations of graphs. The following result is about the admissibility of a strong IASI for the union of two strong IASI graphs.

\begin{theorem}\label{T-SIGU1}
{\rm \cite{GS12}} The union $G_1\cup G_2$ of two graphs $G_1$ and $G_2$, admits a strong IASI if and only if both $G_1$ and $G_2$ admit strong IASIs.
\end{theorem}

\ni Invoking Theorem \ref{T-NN1}, we have

\begin{theorem}\label{T-NNGU1}
{\rm \cite{GS12}} Let $G_1$ and $G_2$ be two strong IASI graphs. Then, $\varkappa(G_1\cup G_2)\ge max\{\varkappa (G_1),\varkappa(G_2)\}$.
\end{theorem}

The following theorem on the existence of a strong IASI by the join of two strong IASI graphs has been established in \cite{GS2} as follows.

\begin{theorem}
{\rm \cite{GS12}} The join of two strong IASI graphs admits a strong IASI if and only if the difference set of the set-label of every vertex of one of them is disjoint from the difference sets of the set- labels of all vertices of the other. 
\end{theorem}

The nourishing number of the join two strong IASI graphs is given in terms of the nourishing numbers of these graphs as

\begin{theorem}
{\rm \cite{GS12}} Let $G_1$ and $G_2$ be two strong IASI graphs. Then,  $\varkappa(G_1+G_2)= \varkappa(G_1)+ \varkappa(G_2)$. 
\end{theorem}

We are now going to review the conditions for the existence of a strong IASI for the complement of a strong IASI graph. Since $G$ and its complement $\overline{G}$ has the same vertex set, the vertices of $G$ and $\overline{G}$ must have the same set-labels. It can be noted that a strong IASI defined on a graph $G$ need not be a strong IASI for $\overline{G}$. Hence, we have proposed  necessary and sufficient condition for the existence of a strong IASI for the complement of a strong IASI graph.

\begin{theorem}
{\rm \cite{GS12}} The complement of a strong IASI graph $G$ admits a strong IASI if and only if the difference sets of the set-labels of all vertices of $G$ are pairwise disjoint.
\end{theorem}


\ni Then, we have

\begin{proposition}
If the complement $\overline{G}$ of a strong IASI graph admits a strong IASI, then $\varkappa(G)=\varkappa(\overline{G})=|V(G)|$. 
\end{proposition}

\begin{corollary}
If $G$ is a self-complementary graph on $n$ vertices, which admits a strong IASI, then $\varkappa(G)=n$.
\end{corollary}

The existence of strong IASI for certain products of graphs have also been established in \cite{GS12}. The admissibility of a strong IASI by the Cartesian product of two strong IASI graphs has been established in \cite{GS12} as

\begin{theorem}
{\rm \cite{GS12}} Let $G_1$ and $G_2$ be two strong IASI graphs. Then, the product $G_1\Box G_2$ admits a strong IASI if and only if the set-labels of corresponding vertices different copies of $G_1$ which are adjacent in $G_1\Box G_2$ are disjoint.   
\end{theorem}

\begin{proposition}
Let $G_1$ and $G_2$ be two graphs which admit strong IASIs. Then, $\varkappa(G_1\Box G_2) = \max\{\varkappa(G_1),\varkappa(G_2)\}$.
\end{proposition}

In \cite{GS12}, the admissibility of a strong IASI by the corona of two strong IASI graphs has been established as follows.

\begin{theorem}
Let $G_1$ and $G_2$ be two strong IASI graphs. Then, their corona $G_1\odot G_2$ admits a strong IASI if and only if the difference set of the set-label of every vertex of $G_1$ is disjoint from the difference sets of the set- labels of all vertices of the corresponding copy of $G_2$. 
\end{theorem}

The nourishing number of the corona of two graphs is given in the following result, in terms of the nourishing number of the individual graphs, as follows.

\begin{proposition}
If $G_1$ and $G_2$ are two strong IASI graphs, then
\[ \varkappa(G_1\odot G_2) = \left\{
  \begin{array}{l l}
    \varkappa(G_1) & \quad \text{if $\varkappa(G_1)> \varkappa(G_2)$}\\
    \varkappa(G_2)+1 & \quad \text{if $\varkappa(G_2)> \varkappa(G_1)$}
  \end{array} \right.\]
\end{proposition}

\section{Associated Graphs of Strong IASI Graphs}

In this section, we review the results on the admissibility of induced strong IASIs by certain graphs that are associated to given strong IASI graphs.

The {\em line graph} of an undirected graph $G$ is another graph $L(G)$ that represents the adjacencies between edges of $G$. That is, every edge of $G$ corresponds to a vertex in $L(G)$ and the adjacency of edges in $G$ defines the adjacency of corresponding vertices in $L(G)$. Hence, we have
 
\begin{theorem}\label{T-SIASI-LG}
{\rm \cite{GS0}} The line graph of a strong IASI graph does not admit an induced strong IASI.
\end{theorem}

The proof of the above theorem follows from the facts that $D_{f(u)}\cup \, D_{f(v)}\subseteq D_{f^+(uv)}$ for any two adjacent vertices $u$ and $v$ in $G$ and for any two adjacent edges $e_1$ and $e_2$ in $G$, $D_{f^+(e_1)}\cap \, D_{f^+(e_2)}\ne \emptyset$.

The {\em total graph}  $T(G)$ of a graph $G$ has a vertex corresponding to each element (edge or vertex) of $G$ such that the adjacency between the vertices in $T(G)$ is determied by the adjacency or incidence between the corresponding elements of $G$>. Then, we have

\begin{theorem}\label{T-SIASI-tG}
{\rm \cite{GS0}} The total graph of a strong IASI graph does not admit an induced strong IASI.
\end{theorem}

The line graph of $G$ is a subgraph of the total graph $G$ and by Theorem \ref{T-SIASI-LG}, $L(G)$ does not admit a strong IASI. Hence, by Theorem \ref{T-SIASI-SG}, $T(G)$ does not admit a strong IASI. 

A graph $H$ is said to be  {\em homeomorphic} to a given graph $G$ if $H$ is obtained by replacing some paths of length $2$ in $G$ which are not in any triangle of $G$, by some edges (which are originally not in $G$). This operation is called an {\em elementary topological reduction}. The following is a necessary and sufficient condition for a graph homeomorphic to a given strong IASI graph $G$ to admit a strong IASI.

\begin{theorem}
{\rm \cite{GS2}} Any graph $H$, obtained by applying a $r$ elementary topological reductions on a strong IASI graph $G$, admits a strong IASI if and only if there exist $r$ distinct paths of length $2$ in $G$ which are not in any triangle of $G$, the difference sets of the set-labels of whose vertices are pairwise disjoint. 
\end{theorem}

A graph $H$ is a {\em minor} of a graph $G$ if it can be obtained from a subgraph of $G$ by contracting some of its edges. Here, contracting an edge $e$ means deleting $e$ and then identifying its end vertices. We now propose the the following results on the admissibility of strong IASI by the minors of given strong IASI graphs.

\begin{theorem}
A minor $H$ of a strong IASI graph admits an induced strong IASI if and only if there exists at least one edge in $G$, the difference set of the set-label of which is disjoint from the difference sets of the set-labels of its neighbouring vertices.
\end{theorem}
\begin{proof}
Let $G$ be a graph tht admits a strong IASI $f$ and let $e=v_iv_j$ be an edge of $G$. Let $H$ be a graph obtained by contracting the edge $e$. If $w$ is the new vertex in $H$ obtained by identifying the end vertices of $e$ in $G$. Then, the set-label of $w$ in $H$ is the same as that of $e$ in $G$. Note that all vertices of $G$ other than $v_i$ and $v_j$ are also the vertices of $H$. Let $v_r$ and $v_s$ be the vertices adjacent to $v_i$ and $v_j$ respetively, in $G$. Then, $v_r$ and $v_s$ be the adjacent vertices of $w$ in $H$. 

Now, assume that $H$ admits a strong IASI, say $f'$, induced by $f$. Then, for all vertices $v\neq v_i,v_j$, we have $f'(v)=f(v)$ and $f'(w)=f^+(uv)$. Therefore, $D_{f'(w)}\cap D_{f(v_r)}=\emptyset$ and $D_{f'(w)}\cap D_{f(v_s)}=\emptyset$ and hence $D_{f^+(v_iv_j)}\cap D_{f(v_r)}=\emptyset$ and $D_{f^+(v_iv_j)}\cap D_{f(v_s)}=\emptyset$.

To prove the converse, assume that $D_{f^+(v_iv_j)}\cap D_{f(v_r)}=\emptyset$ and $D_{f^+(v_iv_j)}\cap D_{f(v_s)}=\emptyset$. Now, let $f'$ be the IASI on $H$ induced by $f$. Then, $f'(w)= f^+(v_iv_j)$ and hence we have $D_{f'(w)}\cap D_{f(v_r)}=\emptyset$ and $D_{f'(w)}\cap D_{f(v_s)}=\emptyset$. Since all other edges $uv$ of $H$ are the edges of $G$ and $f$ is a strong IASI of $G$, we have $D_{f'(u)}\cap D_{f'(u)} = D_{f(u)}\cap D_{f(u)}=\emptyset$. Therefore, $f'$ is a strong IASI of $H$ induced by the strong IASI $f$ of $G$.
\end{proof}

A graph $H$ is a {\em subdivision} of a graph $G$ if $H$ can be obtained from $G$ by replacing some of its edges with some paths of length $2$ such that the inner vertices of these paths have degree $2$ in $H$. The following theorem verifies whether a subdivision of $G$ admits a strong IASI.

\begin{theorem}
A subdivision $S(G)$ of a strong IASI graph $G$ does not admit an induced strong IASI.
\end{theorem}
\begin{proof}
Let the given graph $G$ admits a strong IASI $f$. Since $f$ is a strong IASI of $G$, for any two adjacent vertices $u$ and $v$ in $G$  $D_{f(u)}\cup \, D_{f(v)}\subseteq D_{f^+(uv)}$. Let $P_2:v_iv_jv_k$ be a path in $G$. Now, let $H$ be a graph obtained by replacing the edge $uv$ by a path $uwv$. Here, under an IASI, induced by $f$, the new element $w$ takes the same set-label of the removed edge $uv$. Therefore, $D_{f(u)}\cap \, D_{f(w)} = D_{f(u)}\ne \emptyset$. Similarly, $D_{f(w)}\cap \, D_{f(v)} = D_{f(v)}\ne \emptyset$. Hence, $f$ does not induce a strong IASI on $H$.
\end{proof}

\section{Nourishing Number of Graph Powers}

In this section, we discuss the nourishing number of cetain graph classes and their powers. For a positive integer $r$, the $r$-th power of a simple graph $G$ is the graph $G^r$ whose vertex set is $V$, two distinct vertices being adjacent in $G^r$ if and only if their distance in $G$ is at most $r$. The graph $ G^2 $ is referred to as the {\em square} of $ G $, the graph $ G^3 $ as the {\em cube} of G. To proceed further, we need the following theorem on graph powers.

\begin{theorem}\label{T-Gdiam}
{\rm \cite{EWW}} If $d$ is the diameter of a graph $G$, then $G^d$ is a complete graph.
\end{theorem}

Let us begin with the following result on the nourishing number of the finite powers of complete bipartite graphs.

\begin{theorem}
{\rm \cite{GS13}} The nourishing number of the $r$-th power of a complete bipartite graph is 
 \begin{equation*}
 \varkappa(K_{m,n}^r)=
 	\begin{cases}
 	2 & \text{if}~~ r=1\\
 	m+n & \text{if}~~ r\ge 2.
 	\end{cases}
 \end{equation*}
\end{theorem}

The proof is immediate from the fact that the diameter of a complete bipartite graph is $2$.

\ni The nourishing number of a path has been estimated in \cite{GS13} as follows.
 
\begin{theorem}\label{NNPm^r}
{\rm \cite{GS13}} The nourishing number of the $r$-th power of a path $P_m$ is 
\begin{equation*}
\varkappa(P_m)=
	\begin{cases}
	r+1 & \text{if}~~ r<m\\
	m+1  & \text{if}~~ r\ge m.
	\end{cases}
\end{equation*}
\end{theorem}

The following result is on the nourishing number of the $r$-th power of a cycle.

\begin{theorem}\label{NNCn^r}
{\rm \cite{GS13}} The nourishing number of the $r$-th power of a cycle $C_n$ is 
\begin{equation*}
\varkappa(C_n)=
	\begin{cases}
	r+1 & \text{if}~~ r<\lfloor \frac{n}{2} \rfloor\\
	n  & \text{if}~~ r\ge \lfloor \frac{n}{2} \rfloor.
	\end{cases}
\end{equation*}
\end{theorem}

A {\em wheel graph}, denoted by $W_{n+1}$, is defined as  $W_{n+1}=C_n+K_1$. The nourishing number of the powers of wheel graphs is discussed in the following proposition.

\begin{theorem}
{\rm \cite{GS13}} The nourishing number of the $r$-th power of a wheel graph is 
\begin{equation*}
\varkappa(W_{n+1}^r)=
\begin{cases}
3 & \text{if}~~ r=1\\
n+1 & \text{if}~~ r \ge 2.
\end{cases}
\end{equation*}
\end{theorem}

A {\em helm graph}, denoted by $H_n$, is a graph obtained by adjoining a pendant edge to each of the vertices of the outer cycle of a wheel graph $W_{n+1}$. A helm graph has $2n+1$ vertices and $3n$ edges. The following theorem is on the nourishing number of a helm graph and its powers.

\begin{theorem}\label{T-NNHelmG}
{\rm \cite{GS13}} The nourishing number of the $r$-th power of a helm graph $H_n$ is given by
\begin{equation*}
\varkappa{H_n^r}=
\begin{cases}
3 & \textrm{if}~~~ r=1\\
n+1 & \textrm{if}~~~ r=2\\
n+4 & \textrm{if}~~~ r=3\\
2n+1 & \textrm{if}~~~ r\ge 4.
\end{cases}
\end{equation*}
\end{theorem}

Another graph we consider in this context is a  {\em friendship graph} $F_n$, which is the graph obtained by joining $n$ copies of the cycle  $C_3$ with a common vertex. It has $2n+1$ vertices and $3n$ edges. The following proposition is about the nourishing number of a friendship graph $F_n$.

\begin{theorem}
{\rm \cite{GS13}} The nourishing number of the $r$-th power of a friendship graph $F_n$ is 
\begin{equation*}
\varkappa(F_n^r)=
\begin{cases}
3 & \text{if}~~ r=1\\
2n+1 & \text{if}~~ r\ge 2
\end{cases}
\end{equation*}
\end{theorem}

Another similar graph structure is a {\em fan graph} $F_{m,n}$, which is defined as $\bar{K_m}+P_n$. The nourishing number of a fan graph $F_{m,n}$ and its powers are estimated as follows.

\begin{theorem}
{\rm \cite{GS13}} The nourishing number of the $r$-th power of a fan graph $F_{m,n}$ is 
\begin{equation*}
\varkappa(F_{m,n}^r)=
\begin{cases}
3 & \text{if}~~ r=1\\
m+n & \text{if}~~ r\ge 2
\end{cases}
\end{equation*}
\end{theorem}

A {\em split graph}, denoted by $G(K_r,S)$, is a graph in which the vertices can be partitioned into a clique $K_r$ and an independent set $S$. A split graph is said to be a {\em complete split graph} if every vertex of the independent set $S$ is adjacent to every vertex of the the clique $K_r$ and is denoted by $K_S(r,s)$, where $r$ and $s$ are the orders of $K_r$ and $S$ respectively. The nourishing number of a split graph is estimated in the following theorem.

\begin{theorem}\label{T-NNSpG}
{\rm \cite{GS13}} The nourishing number of the $r$-th power of a split graph $G=G(K_r,S)^r$ is given by
\begin{equation*}
\varkappa(G^r)=
\begin{cases}
r & \text{if no vertex of $S$ is adjacent to all vertices of $K_r$}\\
r+1 & \text{if some vertices of $S$ are adjacent to all vertices of $K_r$}\\
r+l & \text{if $r=2$}\\
r+s & \text{if $r\ge 3$};
\end{cases}
\end{equation*}
\end{theorem}

In view of Theorem \ref{T-NNSpG}, we have established the following result on complete split graphs.

\begin{corollary}
The nourishing number of the $k$-th power of a complete split graph $G=K_S(r,s)$ is given by
\begin{equation*}
\varkappa(G^r)=
\begin{cases}
r+1 & \text{if $k=1$}\\
r+s & \text{if $k\ge 2$};
\end{cases}
\end{equation*}
\end{corollary}

An {\em $n$-sun} or a {\em trampoline}, denoted by $S_n$,  is a chordal graph on $2n$ vertices, where $n\ge 3$, whose vertex set can be partitioned into two sets $U = \{u_1,u_2,u_3,\ldots, u_n\}$ and $W = \{w_1,w_2,w_3,\ldots, w_n\}$ such that $W$ is an independent set of $G$ and $w_j$ is adjacent to $u_i$ if and only if $j=i$ or $j=i+1~(mod ~ n)$. {\em A complete sun} is a sun $G$ where the induced subgraph $\langle U \rangle$ is complete. The sun graphs with $\langle U \rangle$ is a cycle is one of the most interesting structures among the sun graphs.  The following theorem determines the nourishing number of a sun graph $G$ whose non-independent set of vertices induces a cycle in the graph $G$.

\begin{theorem}\label{T-NNSunG}
{\rm \cite{GS13}} If $G$ is an $n$-sun graph with $\langle U \rangle = C_n$, then the nourishing number of $G^r$ is given by
\begin{equation*}
\varkappa(G^r)=
\begin{cases}
2r+1 & \text{if}~~~ r<\lfloor \frac{n}{2}\rfloor\\
2(n-1) & \text{if}~~~ r=\lfloor \frac{n}{2}\rfloor ~\text{and if~ $n$ ~ is odd}\\
2n-1 & \text{if}~~~ r=\lfloor \frac{n}{2}\rfloor ~\text{and if~ $n$ ~ is even}\\
2n & \text{if}~~~ r\ge \lfloor \frac{n}{2}\rfloor +1;
\end{cases}
\end{equation*}
\end{theorem}

\begin{theorem}\label{T-NNKSunG}
{\rm \cite{GS13}} The nourishing number of a complete $n$-sun graph $G$ is given by
\begin{equation*}
\varkappa(G^r)=
\begin{cases}
n & \text{if}~~~ r=1 ~\text{and if}~~~ \langle U \rangle \text{is triangle-free}\\
n+1 & \text{if}~~~ r=2\\
2n & \text{if}~~~ r\ge 3;
\end{cases}
\end{equation*}
\end{theorem}

Another important graph is an {\em $n$-sunlet graph}, which is the graph on $2n$ vertices obtained by attaching one pendant edge to each vertex of a cycle $C_n$. The following result discusses about the nourishing number of an $n$-sunlet graph. 

\begin{theorem}
{\rm \cite{GS13}} The nourishing number of the $r$-th power of an $n$-sunlet graph $G$ is
\begin{equation*}
\varkappa(G^r)=
	\begin{cases}
	2r & \text{if}~~ r<\lfloor \frac{n}{2} \rfloor +1\\
	2(n-1) & \text{if}~~ r= \lfloor \frac{n}{2} \rfloor+1 ~\text{and ~$n$~ is odd}\\ 
	2n-1 & \text{if}~~ r= \lfloor \frac{n}{2} \rfloor+1 ~\text{and ~$n$~ is even}\\
	2n & \text{if}~~ r\ge \lfloor \frac{n}{2} \rfloor+2.
	\end{cases}
\end{equation*} 
\end{theorem}

\section{Strongly Uniform IASI Graphs}

\begin{definition}{\rm
{\rm \cite{GS2}} A graph $G$ is said to have a strongly uniform IASI if there exists an IASI $f$ defined on $G$ such that $|f^+(uv)|=k=|f(u)|\,|f(v)|$ for all $u,v \in V(G)$, where $k$ is a positive integer. }
\end{definition}

In view of the above definition, we can notice that if a graph $G$ admits a strongly $k$-uniform IASI, then every edge of $G$ has the set-indexing number which is the product of the set-indexing numbers of its end vertices. Hence, if $G$ is a strongly $k$-uniform IASI graph, then each vertex of $G$  has some set-indexing number $d_i$, which is a divisor of $k$. Hence, $V(G)$ can be partitioned into at most $n$ sets, say $(X_1, X_2, \ldots, X_n)$ such that each $X_i$ consists of the vertices of $G$ having the set-indexing number $d_i$, where $n$ is the number of divisors of the integer $k$.

The admissibility of strongly uniform IASIs by bipartite graphs had been established in the following theorem.

\begin{theorem}\label{T-SIASI-BG}
{\rm \cite{GS2}} For any positive integer $k$, all bipartite graphs admit a strongly $k$-uniform IASI.
\end{theorem}

The following theorem has established a condition required for a complete graph to admit a strongly $k$-uniform IASI.

\begin{theorem}\label{TCG}
{\rm \cite{GS2}} A strongly $k$- uniform IASI of a complete graph $K_n$  is a $(k,l)$-completely uniform IASI, where $l=\sqrt{k}$.
\end{theorem}

Now, recall the following fact from the number theory. The number of divisors of a non-square integer is even and the number of divisors of a  perfect square is odd. Then, we have

\begin{theorem}\label{T-USIASI1}
{\rm \cite{GS2}} For a positive integer $k$, let $G$ be a graph which admits a strongly $k$-uniform IASI. Also, let $n$ be the number of divisors of $k$. Then, if $k$ is not a perfect square, then $G$ has at most  $\frac{n}{2}$ bipartite components and if  $k$ is a perfect square, then $G$ has at most  of  $\frac{n+1}{2}$ components in which at most $\frac{n-1}{2}$ components are bipartite components.
\end {theorem}

Invoking Theorem \ref{T-USIASI1}, we note that if the vertex set of a graph $G$ admitting a strongly $k$-uniform IASI, can be partitioned into more than two sets, then $G$ is a disconnected graph. Hence, we hve a more generalised result as follows.

\begin{theorem}\label{T-USIASI2}
{\rm \cite{GS2}} Let $k$ be a non-square integer. Then a graph $G$ admits a strongly k-uniform IASI if and only if $G$ is bipartite or a union of disjoint bipartite components.
\end{theorem}

\ni The following results are immediate consequences of the above results.

\begin{theorem}\label{M4}
{\rm \cite{GS2}} If a graph $G$, which admits a strongly $k$-uniform IASI, then it contains at most one component which is a clique.
\end{theorem}

\begin{theorem}\label{M5}
Let the graph $G$ has a strongly $k$-uniform IASI. Hence, if $G$ has a component which is a clique, then $k$ is a perfect square.
\end{theorem}

\ni Hence, we have 

\begin{theorem}\label{TNB}
{\rm \cite{GS2}} A connected non-bipartite graph $G$ admits a strongly $k$-uniform IASI if and only if $k$ is a perfect square and this IASI is a $(k,l)$-completely uniform, where $l=\sqrt{k}$.
\end{theorem}

From all of our discussions aboout the strongly uniform IASI graphs, we can establish a generalised result as follows.

\begin{theorem}
{\rm \cite{GS2}} A connected graph $G$ admits a strongly $k$-uniform IASI if and only if $G$ is bipartite or this IASI is a $(k,l)$-completely uniform IASI of $G$, where $k=l^2$.
\end{theorem}

\section{Conclusion}
In this paper, we have given some characteristics of the graphs which admit strong IASIs and strongly $k$-uniform IASIs. More studies are posssible in this area when the ground set $X$ is finite instead of the set $\mathbb{N}_0$.

The problems of establishing the necessary and sufficient conditions for various graphs and graph classes to admit these types of uniform IASIs have still been unsettled. The existence of strong IASI, both uniform and non-uniform, for certain other graph products is also woth-studying.

Certain IASIs which assign set-labels having specific properties, to the elements of given graphs also seem to be promising for further investigations. All these facts highlight the wide scope for intensive studies in this area.

\end{document}